\documentclass[11pt, letterpaper,reqno]{amsart}

\usepackage{amsmath}
\usepackage{amssymb}
\usepackage{amsfonts}
\usepackage{ dsfont }
\usepackage{cases}
\usepackage{cite}
\usepackage[hyphens]{url}
\usepackage{hyperref}
\usepackage{amsthm}
\usepackage{upgreek}
\usepackage{color}

\usepackage{xcolor}
\definecolor{color0}{gray}{.50}
\definecolor{color1}{rgb}{0,.2,.8}
\definecolor{color2}{rgb}{1,.2,0}
\definecolor{color3}{rgb}{.8,.5,1}

\numberwithin{equation}{section}
\DeclareMathOperator*{\esssup}{ess\,sup}
\DeclareMathOperator*{\essinf}{ess\,inf}

\newtheorem{theorem}{Theorem}[section]

\newtheorem{proposition}[theorem]{Proposition}
\newtheorem{definition}[theorem]{Definition}

\newtheorem{remark}[theorem]{Remark}

\newcommand{\thref}{Theorem~\ref}

\newcommand{\E}{\mathbb{E}}
\newcommand{\PP}{\mathbb{P}}

\newcommand{\R}{\mathbb{R}}
\newcommand{\eps}{\epsilon}

\newcommand{\cT}{\mathcal{T}}
\newcommand{\T}{\mathbb{T}}
\newcommand{\kT}{\mathfrak{T}}
\newcommand{\cF}{\mathcal{F}}

\title[]{Non-zero-sum stopping games in discrete time}
\author[]{Zhou Zhou}
\address{Institute for Mathematics and its Applications, University of Minnesota}
\email{zhouzhou@umich.edu}
\date{August 23, 2015}

\begin{document}
\maketitle
\begin{abstract}
We consider two-player non-zero-sum stopping games in discrete time. Unlike Dynkin games, in our games the payoff of each player is revealed after both players stop. Moreover, each player can adjust her own stopping strategy according to the other player's action. In the first part of the paper, we consider the game where players act simultaneously at each stage. We show that there exists a Nash equilibrium in mixed stopping strategies. In the second part, we assume that one player has to act first at each stage. In this case, we show the existence of a Nash equilibrium in pure stopping strategies.  
\end{abstract}

\section{Introduction}
As a classical model of stopping games, Dynkin game has attracted a lot of research. We refer to \cite{Dynkin,Zhang3,Solan,Hamadene,Kifer,Solan1,Solan2,Ko3,Lepeltier,Solan3,Solan4,Touzi,Neveu,Bismut,Ferenstein} and the references therein. In a Dynkin game, each player chooses a stopping strategy, and the payoffs are revealed when one player stops. In other words, the game ends at the minimum of the stopping strategies. In practice, it is more often that, even if a player has made the decision first, her payoff can still be affected by other players' decisions later on. Therefore, it is more reasonable to let the game end at the maximum of the stopping strategies. Moreover, a wise player would adjust her strategy after she observes other players' actions. Based on these two points, recently \cite{ZZ6,ZZ7,ZZ9} study the stopping games with these features. In particular, \cite{ZZ6,ZZ7} consider the zero-sum case, and \cite{ZZ7} investigates the non-zero-sum case in continuous time.

In this paper, given a filtered probability space $(\Omega,\cF,(\cF)_{t=0,\dotso,T},\PP)$, we consider a non-zero-sum stopping game in discrete time
$$u^i(\rho,\tau)=\E[U^i(\rho,\tau)],\quad i=1,2,$$
where the first (resp. second) player chooses $\rho$ (resp. $\tau$) to maximize the payoff $u^1$ (resp. $u^2$). Here $U^i(s,t)$ is $\cF_{s\vee t}$-measurable instead of $\cF_{s\wedge t}$-measurable as is assumed in Dynkin games. That is, the game ends at the maximum of $\rho$ and $\tau$. Moreover, here $\rho$ and $\tau$ are not (randomized) stopping times, they are strategies that can be adjusted according to each other.

In the first part of the paper, we consider the case when the two players act simultaneously at each stage (here both ``stop'' and ``not stop'' are actions). We show that there exists a perfect Nash equilibrium in mixed stopping strategies. The main idea to prove the result is to convert the original problem to a non-zero-sum Dynkin game with randomized stopping times.

In the second part of the paper, we consider the game where one player acts first at each stage. In this case, we show that there always exists a perfect Nash equilibrium in pure stopping strategies. We use the idea in \cite{ZZ9} to prove this result. That is, we first construct saddle points for some related zero-sum stopping games, and then using these saddle points we construct a Nash equilibrium for the non-zero-sum games. 

This paper extends the result in \cite{ZZ9} to the discrete time case. It has a broad range of applications, e.g., when companies choose times to enter the market, or when investors who both long and short American options choose times to exercise the options. 

The paper is organized as follows. In the next section, we consider the non-zero-sum stopping game when two players act simultaneously at each stage. In Section 3, we study the case when one player acts first at each stage. In Section 4, we make a comparison between our discrete-time results in this paper and the continuous-time result in \cite{ZZ9}.

\section{Stopping games where players act simultaneously at each stage}
In this section, we consider the non-zero-sum stopping game where players act simultaneously at each stage. We will consider mixed stopping strategies for the game. \thref{t1} is the main result of this section.

Let $(\Omega,\cF,(\cF_t)_{t=0,\dotso,T},\PP)$ be a filtered probability space, where $\Omega$ is countably generated, and $T\in\mathbb{N}$ is the finite time horizon. Let $\cT$ be the set of stopping times taking values in $\{0,\dotso,T\}$. For any $\sigma\in\cT$, let $\cT_\sigma:=\{\rho\in\cT,\ \rho\geq\tau\}$, and $\cT_{\sigma+}:=\{\rho\in\cT,\ \rho\geq(\tau+1)\wedge T\}$. Define
$$\T^a:=\{\phi:\{0,\dotso,T\}\times\Omega\mapsto\{0,\dotso,T\}:\ \phi(t,\cdot)\in\cT_{t+},\ t=0,\dotso,T\}.$$
Let $\cT^r$ be the set of randomized stopping times. That is, for any $\alpha\in\cT^r$, $\alpha:[0,1]\times\Omega\mapsto\{0,\dotso,T\}$ is $\mathcal{B}([0,1])\otimes\cF$-measurable, and $\alpha(p,\cdot)\in\cT$ for any $p\in[0,1]$.

\begin{definition}
$(\rho_0,\rho_1)$ is said to be a (pure) stopping strategy of type A, if $\rho_0\in\cT$ and $\rho_1\in\T^a$.
Denote $\kT^a$ as the set of (pure) stopping strategies of type A.
\end{definition}
For $(\rho_0,\rho_1)\in\T^a$, $\rho_0$ represents a player's initial (pure) stopping strategy, and $\rho_1(t,\cdot)$ represents the strategy adjusted by the player after she observes the other player's stopping at time $t$. For $\rho=(\rho_0,\rho_1),\tau=(\tau_0,\tau_1)\in\kT^a$, denote
$$\rho[\tau]:=\rho_0 1_{\{\rho_0\leq\tau_0\}}+\rho_1(\tau_0) 1_{\{\rho_0>\tau_0\}}.$$

\begin{definition}
$(\alpha,\rho_1)$ is said to be a mixed stopping strategy of type A, if $\alpha\in\cT^r$ and $\rho_1\in\T^a$. Denote $\kT^{ar}$ as the set of mixed stopping strategies of type A.
\end{definition}

\begin{remark}
One can also randomize the strategies in $\kT^i$. However, it turns out that we only need to randomize players' initial stopping times (i.e., the first component of $(\rho_0,\rho_1)\in\kT^a)$), in order to get the existence of a Nash equilibrium for the stopping game introduced below.
\end{remark}

For $i=1,2$, let $U^i:\ \{0,\dotso,T\}\times \{0,\dotso,T\}\times\Omega\mapsto\R$, such that $U^i(s,t,\cdot)$ is $\cF_{s\vee t}$-measurable. For simplicity, we assume that $U^i$ is bounded for $i=1,2$. Consider the non-zero-sum stopping game
\begin{equation}\label{e1}
u^i(\rho,\tau)=\int_{[0,1]^2}\Gamma^i\left(\rho(p,\cdot),\tau(q,\cdot)\right)dpdq,\quad\rho,\tau\in\kT^{ar},\quad i=1,2,
\end{equation}
where for $i=1,2$ and $\zeta=(\zeta_0,\zeta_1),\eta=(\eta_0,\eta_1)\in\kT^a$,
\begin{eqnarray}
\notag \Gamma^i(\zeta,\eta)&=&\E\left[U^i(\zeta[\eta],\eta[\zeta])\right]\\
\notag &=&\E\left[U^i(\zeta_0,\eta_1(\zeta_0)) 1_{\{\zeta_0<\eta_0\}}+U^i(\zeta_1(\eta_0),\eta_0) 1_{\{\zeta_0>\eta_0\}}+U^i(\zeta_0,\zeta_0) 1_{\zeta_0=\eta_0\}}\right].
\end{eqnarray}
Here the first player chooses $\rho$ to maximize $u^1$ and the second player chooses $\tau$ to maximize $u^2$. 

Recall the definition of a Nash equilibrium.
\begin{definition}
$(\rho^*,\tau^*)\in(\kT^{ar})^2$ is said to be a Nash equilibrium in $\kT^{ar}$ for the game \eqref{e1}, if for any $\rho,\tau\in\kT^{ar}$,
$$u^1(\rho,\tau^*)\leq u^1(\rho^*,\tau^*)\quad\text{and}\quad u^2(\rho^*,\tau)\leq u^2(\rho^*,\tau^*).$$
\end{definition}
Below is the main result of this section.
\begin{theorem}\label{t1}
There exists a Nash equilibrium in $\kT^{ar}$ for the game \eqref{e1}.
\begin{remark}
We cannot guarantee the existence of a Nash equilibrium for the game \eqref{e1} if players only use pure stopping strategies of type A (i.e., $\kT^a$). Consider the following deterministic one-period example. Let $T=1$ and $u^1(s,t)=-u^2(s,t)=1_{\{s=t\}}$ for $s,t=0,1$. Then it is easy to see that $\kT^a=\{0,1\}$. Obviously there is no Nash equilibrium for \eqref{e1} in $\kT^a$.
\end{remark}
\end{theorem}
\begin{proof}[Proof of \thref{t1}]
For $t=0,\dotso,T$, let
$$Y_t^1:=\esssup_{\sigma\in\cT_{t+}}\E_t[U^1(\sigma,t)]\quad\text{and}\quad X_t^2:=\esssup_{\sigma\in\cT_{t+}}\E_t[U^2(t,\sigma)],$$
where $\E_\theta[\cdot]:=\E[\cdot|\cF_\theta]$ for $\theta\in\cT$. For $t=0,\dotso,T$, let $\rho_1^*(t,\cdot)\in\cT_{t+}$ and $\tau_1^*(t,\cdot)\in\cT_{t+}$ be optimizers for $Y_t^1$ and $X_t^2$ respectively. That is,
$$\E_t[U^1(\rho_1^*(t),t)]=Y_t^1\quad\text{and}\quad\E_t[U^2(t,\tau_1^*(t))]=X_t^2,\quad\text{a.s.}.$$
Obviously $\rho_1^*(\cdot,\cdot),\tau_1^*(\cdot,\cdot)\in\T^a$. For $t=0,\dotso,T$, define
$$X_t^1:=\E_t[U^1(t,\tau_1^*(t))],\quad Y_t^2:=\E_t[U^2(\rho_1^*(t),t)],\quad\text{and}\quad Z_t^i=U^i(t,t),\ i=1,2.$$

Now consider the non-zero-sum Dynkin game with randomized stopping times
\begin{eqnarray}\label{e2}
\tilde u^i(\alpha,\beta)=\int_{[0,1]^2}\left(\E\left[X_\alpha^i 1_{\{\alpha<\beta\}}+Y_\beta^i 1_{\{\alpha>\beta\}}+Z_\alpha^i 1_{\{\alpha=\beta\}}\right]\right)dpdq,\quad i=1,2,
\end{eqnarray}
for $\alpha,\beta\in\cT^r$. By \cite[Theorem 2.1]{Ferenstein}, there exists a Nash equilibrium $(\alpha^*,\beta^*)\in(\cT^r)^2$ for the Dynkin game \eqref{e2}. That is, for any $\alpha,\beta\in\cT^r$,
\begin{equation}\label{e3}
\tilde u^1(\alpha,\beta^*)\leq\tilde u^1(\alpha^*,\beta^*)\quad\text{and}\quad\tilde u^2(\alpha^*,\beta)\leq\tilde u^2(\alpha^*,\beta^*).
\end{equation}

Let $\rho_m^*:=(\alpha^*,\rho_1^*)$ and $\tau_m^*:=(\beta^*,\tau_1^*)$. Now let us show that $(\rho_m^*,\tau_m^*)\in(\kT^{ar})^2$ is a Nash equilibrium for the game \eqref{e1}. Take $\rho=(\alpha,\rho_1)\in\kT^{ar}$. We have that
\begin{eqnarray}
\notag u^1(\rho,\tau_m^*)&=&\int_{[0,1]^2}\left(\E\left[U^1(\alpha,\tau_1^*(\alpha)) 1_{\{\alpha<\beta^*\}}+U^1(\rho_1(\beta^*),\beta^*) 1_{\{\alpha>\beta^*\}}+U^i(\alpha,\alpha) 1_{\{\alpha=\beta^*\}}\right]\right)dpdq\\
\notag &=&\int_{[0,1]^2}\left(\E\left[\E_\alpha[U^1(\alpha,\tau_1^*(\alpha))] 1_{\{\alpha<\beta^*\}}+\E_{\beta^*}[U^1(\rho_1(\beta^*),\beta^*)] 1_{\{\alpha>\beta^*\}}+U^i(\alpha,\alpha) 1_{\{\alpha=\beta^*\}}\right]\right)dpdq\\
\notag &\leq&\int_{[0,1]^2}\left(\E\left[X_\alpha^1 1_{\{\alpha<\beta^*\}}+Y_{\beta^*}^1 1_{\{\alpha>\beta^*\}}+Z_\alpha^1 1_{\{\alpha=\beta^*\}}\right]\right)dpdq\\
\notag &\leq&\int_{[0,1]^2}\left(\E\left[X_{\alpha^*}^1 1_{\{\alpha^*<\beta^*\}}+Y_{\beta^*}^1 1_{\{\alpha^*>\beta^*\}}+Z_{\alpha^*}^1 1_{\{\alpha^*=\beta^*\}}\right]\right)dpdq\\
\notag &=&\int_{[0,1]^2}\left(\E\left[U^1(\alpha^*,\tau_1^*(\alpha)) 1_{\{\alpha^*<\beta^*\}}+U^1(\rho_1^*(\beta^*),\beta^*) 1_{\{\alpha^*>\beta^*\}}+U^i(\alpha^*,\alpha^*) 1_{\{\alpha^*=\beta^*\}}\right]\right)dpdq\\
\notag &=& u^1(\rho_m^*,\tau_m^*),
\end{eqnarray}
where we use \eqref{e3} for the fourth (in)equality. Similarly, we can show that for any $\tau\in\kT^{ar}$,
$$u^2(\rho_m^*,\tau)\leq u^2(\rho_m^*,\tau_m^*).$$
This completes the proof of the result.
\end{proof}

\section{Stopping games where one player acts first at each stage}

In this section, we consider the stopping games in which one player acts first at each stage. We show that there always exists a Nash equilibrium in pure stopping strategies. \thref{t2} is the main result of this section.

Let
$$\T^b:=\{\psi:\{0,\dotso,T\}\times\Omega\mapsto\{0,\dotso,T\}:\ \psi(t,\cdot)\in\cT_t\}.$$
Here $\psi\in\T^b$ represents a player's (player 2) strategy adjusted at the time when the other player (player 1) stops. In other words,  player 1 acts first at each stage. (Compare $\T^b$ with $\T^a$.)
\begin{definition}
$(\tau_0,\tau_1)$ is said to be a (pure) stopping strategy of type B, if $\tau_0\in\cT$ and $\tau_1\in\T^b$. Denote $\kT^b$ as the set of (pure) stopping strategies of type B.
\end{definition}

For any $\rho=(\rho_0,\rho_1)\in\kT^a,\tau=(\tau_0,\tau_1)\in\kT^b$, 
$$\rho\langle\tau\rangle:=\rho_0 1_{\{\rho_0\leq\tau_0\}}+\rho_1(\tau_0) 1_{\{\rho_0>\tau_0\}}\quad\text{and}\quad\tau\langle\rho\rangle:=\tau_0 1_{\{\tau_0<\rho_0\}}+\tau_1(\rho_0) 1_{\{\tau_0\geq\rho_0\}}.$$

Consider the non-zero-sum stopping game
\begin{equation}\label{e6}
w^i(\rho,\tau):=\E\left[U^i(\rho\langle\tau\rangle,\tau\langle\rho\rangle)\right]=\E\left[U^i(\rho_0,\tau_1(\rho_0))1_{\{\rho_0\leq\tau_0\}}+U^i(\rho_1(\tau_0),\tau_0)1_{\{\rho_0>\tau_0\}}\right],
\end{equation}
for $\rho=(\rho_0,\rho_1)\in\kT^a$, $\tau=(\tau_0,\tau_1)\in\kT^b$ and $i=1,2$.
\begin{definition}
$(\rho^*,\tau^*)\in\kT^a\times\kT^b$ is said to be a Nash equilibrium for the game \eqref{e6}, if for any $\rho\in\kT^a$ and $\tau\in\kT^b$,
$$w^1(\rho,\tau^*)\leq w^1(\rho^*,\tau^*)\quad\text{and}\quad w^2(\rho^*,\tau)\leq w^2(\rho^*,\tau^*).$$
\end{definition}
Below is the main result of this section.
\begin{theorem}\label{t2}
There exists a Nash equilibrium for the game \eqref{e6}.
\end{theorem}

We will use the idea in \cite{ZZ9} to prove \thref{t2}. To be more specific, we will use the saddle points of some zero-sum stopping games to construct a Nash equilibrium for the non-zero-sum game \eqref{e6}. We will first provide some results in the zero-sum case in Section 3.1. Then we prove \thref{t2} in Section 3.2.

\subsection{Zero-sum case} We consider the stopping game in the zero-sum case, i.e., when $U^1=-U^2=U$. We will construct a saddle point for the zero-sum game. (The results in this section are essentially provided in \cite{ZZ6}. We present them for the completeness of this paper.)

For any $\sigma\in\cT$, consider the zero-sum stopping game
\begin{equation}\label{e8}
\underline v_\sigma:=\esssup_{\rho\in\kT_\sigma^a}\essinf_{\tau\in\kT_\sigma^b}\E_\sigma\left[U(\rho\langle\tau\rangle,\tau\langle\rho\rangle)\right],
\end{equation}
and 
\begin{equation}\label{e9}
\overline v_\sigma:=\essinf_{\tau\in\kT_\sigma^b}\esssup_{\rho\in\kT_\sigma^a}\E_\sigma\left[U(\rho\langle\tau\rangle,\tau\langle\rho\rangle)\right],
\end{equation}
where
$$\kT_\sigma^a:=\{(\rho_0,\rho_1)\in\kT^a:\ \rho_0\geq\sigma\}\quad\text{and}\quad\kT_\sigma^b:=\{(\rho_0,\rho_1)\in\kT^b:\ \rho_0\geq\sigma\}.$$

For $t=0,\dotso,T$, let
$$F_t:=\essinf_{\xi\in\cT_t}\E_t[U(t,\xi)]\quad\text{and}\quad G_t:=\left(\esssup_{\xi\in\cT_{t+}}\E_t[U(\xi,t)]\right)\vee F_t,$$
and $\tilde\rho(t,\cdot)\in\cT_{t+}$ and $\tilde\tau(t,\cdot)\in\cT$ be optimizers for $\esssup_{\xi\in\cT_{t+}}\E_t[U(\xi,t)]$ and $F_t$ respectively. Since $F\leq G$, we have that
\begin{equation}\label{e7}
v_\sigma:=\esssup_{\rho\in\cT_\sigma}\essinf_{\tau\in\cT_\sigma}\E_\sigma\left[F_\rho 1_{\{\rho\leq\tau\}}+G_\tau 1_{\{\rho>\tau\}}\right]=\essinf_{\tau\in\cT_\sigma}\esssup_{\rho\in\cT_\sigma}\E_\sigma\left[F_\rho 1_{\{\rho\leq\tau\}}+G_\tau 1_{\{\rho>\tau\}}\right],
\end{equation}
and $(\rho_\sigma,\tau_\sigma)$ is a saddle point for the Dynkin game \eqref{e7}, where
$$\rho_\sigma:=\inf\{t\geq\sigma:\ v_t=F_t\}\quad\text{and}\quad\tau_\sigma:=\inf\{t\geq\sigma:\ v_t=G_t\}.$$
That is, for any $\rho,\tau\in\cT_\sigma$,
$$\E_\sigma\left[F_\rho 1_{\{\rho\leq\tau_\sigma\}}+G_{\tau_\sigma} 1_{\{\rho>\tau_\sigma\}}\right]\leq\E_\sigma\left[F_{\rho_\sigma} 1_{\{\rho_\sigma\leq\tau_\sigma\}}+G_{\tau_\sigma} 1_{\{\rho_\sigma>\tau_\sigma\}}\right]\leq\E_\sigma\left[F_{\rho_\sigma} 1_{\{\rho_\sigma\leq\tau\}}+G_\tau 1_{\{\rho_\sigma>\tau\}}\right].$$
Let $\rho_\sigma^*:=(\rho_\sigma,\tilde\rho)\in\kT^a$ and $\tau_\sigma^*:=(\tau_\sigma,\tilde\tau)\in\kT^b$.
\begin{proposition}
We have $\overline v_\sigma=\underline v_\sigma=v_\sigma$. Moreover, $(\rho_\sigma^*,\tau_\sigma^*)$ is a saddle point of the game \eqref{e8} and \eqref{e9}.
\end{proposition}
\begin{proof}
Take $\tau=(\tau_0,\tau_1)\in\kT^b$. We have that
\begin{eqnarray}
\notag \E_\sigma\left[U(\rho_\sigma^*\langle\tau\rangle,\tau\langle\rho_\sigma^*\rangle)\right]&=&\E_\sigma\left[U(\rho_\sigma,\tau_1(\rho_\sigma))1_{\{\rho_\sigma\leq\tau_0\}}+U(\tilde\rho(\tau_0),\tau_0)1_{\{\rho_\sigma>\tau_0\}}\right]\\
\notag &=&\E_\sigma\left[\E_{\rho_\sigma}[U(\rho_\sigma,\tau_1(\rho_\sigma))]1_{\{\rho_\sigma\leq\tau_0\}}+\E_{\tau_0}[U(\tilde\rho(\tau_0),\tau_0)]1_{\{\rho_\sigma>\tau_0\}}\right]\\
\notag &\geq&\E_\sigma\left[F_{\rho_\sigma} 1_{\{\rho_\sigma\leq\tau_0\}}+G_{\tau_0} 1_{\{\rho_\sigma>\tau_0\}}\right]\\
\notag &\geq&\E_\sigma\left[F_{\rho_\sigma} 1_{\{\rho_\sigma\leq\tau_\sigma\}}+G_{\tau_\sigma} 1_{\{\rho_\sigma>\tau_\sigma\}}\right]\\
\notag &=&\E_\sigma\left[U(\rho_\sigma,\tilde\tau(\rho_\sigma))1_{\{\rho_\sigma\leq\tau_\sigma\}}+U(\tilde\rho(\tau_\sigma),\tau_\sigma)1_{\{\rho_\sigma>\tau_\sigma\}}\right]\\
\notag &=&\E_\sigma\left[U(\rho_\sigma^*\langle\tau_\sigma^*\rangle,\tau_\sigma^*\langle\rho_\sigma^*\rangle)\right],
\end{eqnarray}
where for the third and sixth (in)equalities we use the fact that, on $\{t<\rho_\sigma\}$, $G_t\geq v_t>F_t$, and thus $G_t=\esssup_{\xi\in\cT_{t+}}\E_t[U(\xi,t)]=\E_t[U(\tilde\rho(t),t)]$.

Similarly, we can show that for any $\rho\in\kT^a$,
$$\E_\sigma\left[U(\rho\langle\tau_\sigma^*\rangle,\tau_\sigma^*\langle\rho\rangle)\right]\leq\E_\sigma\left[U(\rho_\sigma^*\langle\tau_\sigma^*\rangle,\tau_\sigma^*\langle\rho_\sigma^*\rangle)\right].$$
This completes the proof of the result.
\end{proof}

\subsection{Proof of \thref{t2}} We will use the saddle points of some related zero-sum stopping games to construct a Nash equilibrium for the non-zero-sum game \eqref{e6}.

For $t=0,\dotso,T$, let
$$F_t^1:=\essinf_{\xi\in\cT_t}\E_t[U^1(t,\xi)]\quad\text{and}\quad G_t^1:=\left(\esssup_{\xi\in\cT_{t+}}\E_t[U^1(\xi,t)]\right)\vee F_t^1,$$
$$F_t^2:=\esssup_{\xi\in\cT_t}\E_t[U^2(t,\xi)]\quad\text{and}\quad G_t^2:=\left(\essinf_{\xi\in\cT_{t+}}\E_t[U^2(\xi,t)]\right)\wedge F_t^2,$$
and
\begin{equation}\label{e11}
v_t^1:=\essinf_{\tau\in\cT_t}\esssup_{\rho\in\cT_t}\E_t\left[F_t^1 1_{\{\rho\leq\tau\}}+G_t^1 1_{\{\rho>\tau\}}\right]=\esssup_{\rho\in\cT_t}\essinf_{\tau\in\cT_t}\E_t\left[F_t^1 1_{\{\rho\leq\tau\}}+G_t^1 1_{\{\rho>\tau\}}\right],
\end{equation}
\begin{equation}\label{e12}
v_t^2:=\esssup_{\tau\in\cT_t}\essinf_{\rho\in\cT_t}\E_t\left[F_t^2 1_{\{\rho\leq\tau\}}+G_t^2 1_{\{\rho>\tau\}}\right]=\essinf_{\rho\in\cT_t}\esssup_{\tau\in\cT_t}\E_t\left[F_t^2 1_{\{\rho\leq\tau\}}+G_t^2 1_{\{\rho>\tau\}}\right].
\end{equation}
Let $\tilde\tau^1(t), \tilde\rho^1(t),\tilde\tau^2(t),\tilde\rho^2(t)$ be optimizers for $F_t^1,\esssup_{\xi\in\cT_{t+}}\E_t[U^1(\xi,t)],F_t^2,\essinf_{\xi\in\cT_{t+}}\E_t[U^2(\xi,t)]$ respectively. Let
$$H_t^1:=\E_t[U^1(t,\tilde\tau^2(t)]\quad\text{and}\quad H_t^2:=\E_t[U^2(\tilde\rho^1(t),t)]$$
for $t=0,\dotso,T$, and
$$\mu^1:=\inf\{t\geq 0:\ v_t^1\leq H_t^1\}\quad\text{and}\quad\mu^2:=\inf\{t\geq 0:\ v_t^2\leq H_t^2\wedge F_t^2\},$$
and
\begin{equation}\label{e15}
\rho_{\mu^2}^2:=\inf\{t\geq\mu^2:\ v_t^2=F_t^2\}\quad\text{and}\quad\tau_{\mu^1}^1:=\inf\{t\geq\mu^1:\ v_t^1=G_t^1\}.
\end{equation}
Define
\[ \rho_0^* = \begin{cases} 
      \mu^1, & \text{if }\mu^1\leq\mu^2,\\
      \rho_{\mu^2}^2, & \text{if }\mu^1>\mu^2,\\
\end{cases}
\quad\quad
\rho_1^*(t) = \begin{cases} 
      \tilde\rho^2(t), & \text{if }t\geq\mu^2+1\text{ and }\mu^1>\mu^2,\\
      \tilde\rho^1(t), & \text{otherwise},\\
\end{cases} \]
\[ \tau_0^* = \begin{cases} 
      \tau_{\mu^1}^1, & \text{if }\mu^1\leq\mu^2,\\
      \mu^2, & \text{if }\mu^1>\mu^2,\\
\end{cases}
\quad\quad
\tau_1^*(t) = \begin{cases} 
      \tilde\tau^1(t), & \text{if }t\geq\mu^1+1\text{ and }\mu^1\leq\mu^2,\\
      \tilde\tau^2(t), & \text{otherwise},\\
\end{cases} \]
for $t=0,\dotso,T$, and
$$\rho^*:=(\rho_0^*,\rho_1^*)\quad\text{and}\quad\tau^*:=(\tau_0^*,\tau_1^*).$$
It can be shown that $\rho^*\in\kT^a$ and $\tau^*\in\kT^b$.

\begin{proposition}
$(\rho^*,\tau^*)$ is a Nash equilibrium for the game \eqref{e6}. Therefore, \thref{t2} holds.
\end{proposition}
\begin{proof}
\textbf{Part 1}: We will show that
\begin{equation}\label{e14}
w^1(\rho,\tau^*)\leq w^1(\rho^*,\tau^*)
\end{equation}
for any $\rho\in\kT^a$. As $F^1\leq H^1$,
\begin{equation}\label{e13}
\mu^1\leq\rho_0^1:=\inf\{t\geq 0:\ v_t^1=F_t^1\}.
\end{equation}
Hence, on $\{t<\mu^1\}$ we have that $G_t^1\geq v_t^1>F_t^1$, and thus $G_t^1=\esssup_{\xi\in\cT_{t+}}\E_t[U^1(\xi,t)]=\E_t[U^1(\tilde\rho^1(t),t)]$. Then
$$w^1(\rho^*,\tau^*)=\E\left[U^1(\mu^1,\tilde\tau^2(\mu^1))1_{\{\mu^1\leq\mu^2\}}+U^1(\tilde\rho^1(\mu^2),\mu^2)1_{\{\mu^1>\mu^2\}}\right]=\E\left[H_{\mu^1}^1 1_{\{\mu^1\leq\mu^2\}}+G_{\mu^2}^1 1_{\{\mu^1>\mu^2\}}\right].$$
Now take $\rho=(\rho_0,\rho_1)\in\kT^a$ and consider $w^1(\rho,\tau^*)$. We will consider four cases.

\textbf{Case 1.1}: $A_1:=\{\rho_0<\mu^1\wedge\mu^2\}$. Since $\mu^1\leq\rho_0^1$ by \eqref{e13}, the process $(v_{t\wedge\mu^1})_{t=0,\dotso,T}$ is a sub-martingale. Then
\begin{eqnarray}
\notag \E\left[U^1(\rho\langle\tau^*\rangle,\tau^*\langle\rho\rangle)1_{A_1}\right]&=&\E\left[U^1(\rho_0,\tilde\tau^2(\rho_0))1_{A_1}\right]\\
\notag &=&\E\left[H_{\rho_0}^11_{A_1}\right]\\
\notag &\leq&\E\left[v_{\rho_0}^11_{A_1}\right]\\
\notag &=&\E\left[v_{\rho_0\wedge\mu^1\wedge\mu^2}^11_{A_1}\right]\\
\notag &\leq&\E\left[\E_{\rho_0\wedge\mu^1\wedge\mu^2}\left[v_{\mu^1\wedge\mu^2}^1\right]1_{A_1}\right]\\
\notag &=&\E\left[\left(v_{\mu^1}^1 1_{\{\mu^1\leq\mu^2\}}+v_{\mu^2}^1 1_{\{\mu^1>\mu^2\}}\right)1_{A_1}\right]\\
\notag &\leq&\E\left[\left(H_{\mu^1}^1 1_{\{\mu^1\leq\mu^2\}}+G_{\mu^2}^1 1_{\{\mu^1>\mu^2\}}\right)1_{A_1}\right].
\end{eqnarray}

\textbf{Case 1.2}: $A_2:=\{\rho_0=\mu^1\wedge\mu^2\}$. We have that
\begin{eqnarray}
\notag \E\left[U^1(\rho\langle\tau^*\rangle,\tau^*\langle\rho\rangle)1_{A_2}\right]&=&\E\left[U^1(\rho_0,\tilde\tau^2(\rho_0))1_{A_2}\right]\\
\notag &=&\E\left[H_{\rho_0}^1 1_{A_2}\right]\\
\notag &=&\E\left[\left(H_{\mu^1}^1 1_{\{\mu^1\leq\mu^2\}}+H_{\mu^2}^1 1_{\{\mu^1>\mu^2\}}\right) 1_{A_2}\right]\\
\notag &\leq&\E\left[\left(H_{\mu^1}^1 1_{\{\mu^1\leq\mu^2\}}+v_{\mu^2}^1 1_{\{\mu^1>\mu^2\}}\right) 1_{A_2}\right]\\
\notag &\leq&\E\left[\left(H_{\mu^1}^1 1_{\{\mu^1\leq\mu^2\}}+G_{\mu^2}^1 1_{\{\mu^1>\mu^2\}}\right) 1_{A_2}\right].
\end{eqnarray}

\textbf{Case 1.3}: $A_3:=\{\rho_0>\mu^1\wedge\mu^2\}\cap\{\mu^1\leq\mu^2\}$. Let $\hat\tau^*:=(\tau_{\mu^1}^1,\tilde\tau^1)\in\kT^b$. We have that
\begin{eqnarray}
\notag \E\left[U^1(\rho\langle\tau^*\rangle,\tau^*\langle\rho\rangle)1_{A_3}\right]&=&\E\left[U^1(\rho\langle\hat\tau^*\rangle,\hat\tau^*\langle\rho\rangle)1_{A_3}\right]\\
\notag &=&\E\left[\E_{\mu^1}\left[U^1(\rho\langle\hat\tau^*\rangle,\hat\tau^*\langle\rho\rangle)\right]1_{A_3}\right]\\
\notag &\leq&\E\left[v_{\mu^1}^1 1_{A_3}\right]\\
\notag &\leq&\E\left[H_{\mu^1}^1 1_{A_3}\right]\\
\notag &=&\E\left[\left(H_{\mu^1}^1 1_{\{\mu^1\leq\mu^2\}}+G_{\mu^2}^1 1_{\{\mu^1>\mu^2\}}\right) 1_{A_3}\right].
\end{eqnarray}

\textbf{Case 1.4}: $A_4:=\{\rho_0>\mu^1\wedge\mu^2\}\cap\{\mu^1>\mu^2\}$.
\begin{eqnarray}
\notag \E\left[U^1(\rho\langle\tau^*\rangle,\tau^*\langle\rho\rangle)1_{A_4}\right]&=&\E\left[U^1(\rho_1(\mu^2),\mu^2)1_{A_4}\right]\\
\notag &\leq&\E\left[G_{\mu^2}^1 1_{A_4}\right]\\
\notag &=&\E\left[\left(H_{\mu^1}^1 1_{\{\mu^1\leq\mu^2\}}+G_{\mu^2}^1 1_{\{\mu^1>\mu^2\}}\right) 1_{A_4}\right].
\end{eqnarray}

By cases 1.1-1.4, we have \eqref{e14} holds.

\textbf{Part 2}: We will show that
\begin{equation}\label{e14}
w^2(\rho^*,\tau)\leq w^2(\rho^*,\tau^*)
\end{equation}
for any $\tau\in\kT^b$. We have that
$$w^2(\rho^*,\tau^*)=\E\left[U^2(\mu^1,\tilde\tau^2(\mu^1))1_{\{\mu^1\leq\mu^2\}}+U^1(\tilde\rho^1(\mu^2),\mu^2)1_{\{\mu^1>\mu^2\}}\right]=\E\left[F_{\mu^1}^2 1_{\{\mu^1\leq\mu^2\}}+H_{\mu^2}^2 1_{\{\mu^1>\mu^2\}}\right].$$
Take $\tau=(\tau_0,\tau_1)\in\kT^b$ and consider $w^2(\rho^*,\tau)$. We will consider five cases.

\textbf{Case 2.1}: $B_1:=\{\tau_0<\mu^1\wedge\mu^2\}$. On $\{t<\mu^2\}$, $F_t^2\geq v_t^2>H_t^2\wedge F_t^2$, and thus $v_t^2>H_t^2\wedge F_t^2=H_t^2$. Moreover, since $H^2\wedge F^2\geq G^2$,
$$\mu^2\leq\inf\{t\geq 0:\ v_t^2= G_t^2\}.$$
Hence, the process $(v_{t\wedge\mu^2}^2)_{t=0,\dotso,T}$ is a sub-martingale. Then following the argument in the case 1.1, we can show that
$$\E\left[U^2(\rho^*\langle\tau\rangle,\tau\langle\rho^*\rangle)1_{B_1}\right]\leq\E\left[\left(F_{\mu^1}^2 1_{\{\mu^1\leq\mu^2\}}+H_{\mu^2}^2 1_{\{\mu^1>\mu^2\}}\right)1_{B_1}\right].$$

\textbf{Case 2.2}: $B_2:=\{\tau_0=\mu^1\wedge\mu^2\}\cap\{\mu^1>\mu^2\}\cap\{\rho_{\mu^2}^2=\mu^2\}$. We have that
\begin{eqnarray}
\notag \E\left[U^2(\rho^*\langle\tau\rangle,\tau\langle\rho^*\rangle)1_{B_2}\right]&=&\E\left[U^2(\mu^2,\tau_1(\mu^2))1_{B_2}\right]\\
\notag &\leq&\E\left[F_{\mu^2}^2 1_{B_2}\right]\\
\notag &=&\E\left[v_{\mu^2}^2 1_{B_2}\right]\\
\notag &\leq&\E\left[H_{\mu^2}^2 1_{B_2}\right]\\
\notag &=&\E\left[\left(F_{\mu^1}^2 1_{\{\mu^1\leq\mu^2\}}+H_{\mu^2}^2 1_{\{\mu^1>\mu^2\}}\right)1_{B_2}\right],
\end{eqnarray}
where the third (in)equality follows from the definition of $\rho_{\mu^2}^2$ in \eqref{e15}.

\textbf{Case 2.3}: $B_3:=\{\tau_0=\mu^1\wedge\mu^2\}\setminus(\{\mu^1>\mu^2\}\cap\{\rho_{\mu^2}^2=\mu^2\})$. We have that
\begin{eqnarray}
\notag \E\left[U^2(\rho^*\langle\tau\rangle,\tau\langle\rho^*\rangle)1_{B_3}\right]&=&\E\left[\left(U^2(\mu^1,\tau_1(\mu^1))1_{\{\mu^1\leq\mu^2\}}+U^2(\tilde\rho^1(\mu^2),\mu^2)1_{\{\mu^1>\mu^2\}}\right)1_{B_2}\right]\\
\notag &\leq&\E\left[\left(F_{\mu^1}^2 1_{\{\mu^1\leq\mu^2\}}+H_{\mu^2}^2 1_{\{\mu^1>\mu^2\}}\right)1_{B_3}\right].
\end{eqnarray}

\textbf{Case 2.4}: $B_4:=\{\tau_0>\mu^1\wedge\mu^2\}\cap\{\mu^1\leq\mu^2\}$. Following the argument in case 1.4, we can show that
$$\E\left[U^2(\rho^*\langle\tau\rangle,\tau\langle\rho^*\rangle)1_{B_4}\right]\leq\E\left[\left(F_{\mu^1}^2 1_{\{\mu^1\leq\mu^2\}}+H_{\mu^2}^2 1_{\{\mu^1>\mu^2\}}\right)1_{B_4}\right].$$

\textbf{Case 2.5}: $B_5:=\{\tau_0>\mu^1\wedge\mu^2\}\cap\{\mu^1>\mu^2\}$. Following the argument in case 1.3, we can show that
$$\E\left[U^2(\rho^*\langle\tau\rangle,\tau\langle\rho^*\rangle)1_{B_5}\right]\leq\E\left[\left(F_{\mu^1}^2 1_{\{\mu^1\leq\mu^2\}}+H_{\mu^2}^2 1_{\{\mu^1>\mu^2\}}\right)1_{B_5}\right].$$
From cases 2.1-2.5, we have \eqref{e14} holds.
\end{proof}

\section{Comparison with the result in \cite{ZZ9}}
In this paper, regarding the existence of a Nash equilibrium in pure strategies, it leads to different results whether players act simultaneously or not at each stage. 

It can be expected that, in continuous time, as long as we have enough regularity for related processes, we would have the existence of an ($\eps$-) Nash equilibrium for the stopping game where one player acts first at each time. Unlike the case in discrete time, if we impose some (right) continuity assumption of $U^i$ in $(s,t)$, then it would not make too much difference whether players act simultaneously or not. Indeed, in \cite{ZZ9}, by assuming the continuity of $U^i$ in $(s,t)$, we show the existence of an $\eps$-Nash equilibrium in pure strategies for any $\eps>0$ for the stopping game in continuous time, where players act simultaneously at each time.

\bibliographystyle{siam}
\bibliography{ref}

\end{document}